\documentclass{amsart}
\usepackage[utf8]{inputenc}
\usepackage[english]{babel}
\usepackage{amsmath, amssymb, amsthm, amsfonts}
\usepackage{geometry}
\usepackage{hyperref}
\usepackage{xcolor}
\usepackage{enumitem}
\usepackage{tikz-cd}
\usetikzlibrary{babel}
\numberwithin{equation}{subsection}
\newtheorem{teorema}{Theorem}[section]
\theoremstyle{definition}
\newtheorem{definicion}[teorema]{Definition}

\newtheorem{observacion}[teorema]{Remark}

\theoremstyle{plain}
\newtheorem{lema}[teorema]{Lemma}
\newtheorem{proposicion}[teorema]{Proposition}

\newtheorem{conjetura}[teorema]{Conjecture}

\newcommand{\Q}{\mathbb{Q}}
\newcommand{\Z}{\mathbb{Z}}
\newcommand{\C}{\mathbb{C}}
\newcommand{\R}{\mathbb{R}}

\newcommand{\cO}{\mathcal{O}}

\DeclareMathOperator{\Nrd}{Nrd}       

\DeclareMathOperator{\End}{End}

\DeclareMathOperator{\Trd}{Trd}       

\hypersetup{
    colorlinks=true,
    linkcolor=blue,
    citecolor=blue,
    urlcolor=blue
}
\title{A lower bound for the distance between CM points on Shimura curves}
\author{Daniel Rodriguez}
\address{Facultad de Matemáticas, Pontificia Universidad Católica de Chile, Avenida Vicuña Mackenna 4860, Santiago, Chile.}
\email{dcrodriguez1@uc.cl}
\date{\today}

\begin{document}

\begin{abstract}
In this paper, we establish a quantitative Diophantine approximation result for complex multiplication (CM) points on Shimura curves. Specifically, we prove a lower bound for the distance between a sequence of CM points $P_n$ converging to a fixed CM point $P$ on a Shimura curve $X(D,1)$ in terms of the discriminant of the endomorphism rings of $P_n$. The proof exploits the complex geometry of the Fuchsian uniformization, the explicit matrix representation of the underlying quaternion algebra, and Liouville's inequality. We show that the distance between the corresponding fixed points $\tau_n$ and $\tau$ in the upper half-plane is bounded below by a positive constant times a negative power of the discriminant. This result provides a Shimura curve analogue of a result of Habegger on singular moduli and modular curves.
\end{abstract}

\maketitle

\section{Introduction}

Special points play a central role in arithmetic geometry due to their deep arithmetic significance within algebraic varieties. Classical examples include torsion points on semiabelian varieties, preperiodic points in algebraic dynamical systems, and complex multiplication (CM) points on moduli spaces of abelian varieties. In this work, we focus on the approximation of algebraic points by special points.

A landmark classical result in this setting is Baker's theorem on linear forms in logarithms, which provides an explicit lower bound for the absolute value of linear combinations of logarithms of algebraic numbers (see \cite[Theorem 3.1]{Baker}). This theorem implies that an algebraic point on the unit circle that is not a root of unity cannot be approximated too quickly by roots of unity (see \cite[Proposition 1.2]{Baker-Il-Rumely}).

In the non-archimedean setting, let $K$ be a field complete with respect to a non-archimedean absolute value. In \cite{Tate-Voloch}, Tate and Voloch conjectured a uniform lower bound for the $p$-adic distance from a torsion point on a semiabelian variety to a closed subvariety, provided the torsion point does not lie in the subvariety.

\begin{conjetura}
Let $A/K$ be a semiabelian variety and let $X \subset A$ be a closed subvariety. Then there exists $c > 0$ such that for every torsion point $P \in A(K)$, either $P \in X$ or $d(P,X) \ge c$.
\end{conjetura}

This conjecture was established for algebraic tori when $K = \mathbb{C}_p$ as a consequence of the following theorem on linear forms in roots of unity.

\begin{teorema}\cite[Theorem 2]{Tate-Voloch}
For every integer $n\geq 1$ and every family of $n$ elements $a_1,\dots,a_n \in \mathbb{C}_p$, there exists $c > 0$ such that for any roots of unity $\zeta_1,\dots,\zeta_n \in \mathbb{C}_p$, either $$\displaystyle\sum_{i=1}^n a_i\zeta_i = 0 \quad \text{ or } \quad \left|\displaystyle\sum_{i=1}^n a_i\zeta_i\right| \ge c.$$
\end{teorema}

Moving from semiabelian varieties to moduli spaces, and still in the non-archimedean setting, Habegger (see \cite[Theorem 1]{Habegger2}) was the first to propose an analogue of this phenomenon for moduli spaces by replacing torsion points with CM points of ordinary reduction, proving the result for products of modular curves. Subsequently, Qiu \cite{Qiu} established a much more general result for products of Siegel moduli spaces with CM points of ordinary reduction.

In a related archimedean direction, Habegger proved that if $j(\eta)$ is a fixed algebraic number, then there exist positive constants $A$ and $B$ such that for any elliptic curve $E$ with complex multiplication and $j(E) \neq j(\eta)$, 
$$|j(E)-j(\eta)|\geq \frac{A}{|\Delta_E|^B},$$
where $\Delta_E=\operatorname{disc}(\operatorname{End}(E))$ (see \cite[Lemmas 5 and 8, formula (11)]{Habegger}). The non-archimedean counterpart of this result, under the additional hypothesis that $\eta$ is quadratic imaginary, was later established by Herrero, Menares, and Rivera-Letelier (see \cite[Proposition 4.1]{Herrero_Menares}).

The aforementioned developments strongly motivate the present work, in which we establish an analogue of Habegger's Diophantine approximation result within the framework of Shimura curves.

\begin{teorema}\label{mainthm}
Let $B$ be an indefinite quaternion algebra over $\mathbb{Q}$ with reduced discriminant $D>1$. Let $\mathcal{O}$ be a maximal order in $B$ and let $X(D,1) = \Gamma(\mathcal{O},1)\backslash \mathcal{H}$ be the associated Shimura curve. Let $P = [(A,\iota)]$ and $P_n = [(A_n,\iota_n)]$ be CM points in $X(D,1)$ such that $P_n \to P$ as $n\to\infty$. Then there exists a positive constant $\kappa$ such that
\[
|P_n - P| \;\ge\; \frac{\kappa}{|\operatorname{disc}(P_n)|^{4}},
\]
where $\operatorname{disc}(P_n) = \operatorname{disc}(\operatorname{End}(A_n,\iota_n))$.
\end{teorema}

In outline, the proof exploits the interplay between the complex geometry of the Fuchsian uniformization of $X(D,1)$ and the arithmetic of optimal embeddings of quadratic orders into the maximal order $\mathcal{O} \subset B$. Each CM point $P_n$ corresponds to a fixed point $\tau_n \in \mathcal{H}$ of the image of an imaginary quadratic field $K_n$ under an embedding $\phi_n: K_n \hookrightarrow B$. Using the explicit matrix representation $\varphi: B \hookrightarrow M_2(\mathbb{R})$, we express $\tau_n$ as a rational function of the coefficients $y_n,z_n,t_n$ describing $\phi_n(\sqrt{d_n})$ (where $d_n$ is the discriminant of $K_n$ and $f_n$ is the conductor of $R_n$). From the norm condition $D \frac{y_n^2}{N^2 f_n^2} + q \frac{z_n^2}{N^2 f_n^2} - D q \frac{t_n^2}{N^2 f_n^2} = d_n$ and the convergence $\tau_n \to \tau$, we derive uniform bounds $|y_n|,|z_n|,|t_n| \le C f_n\sqrt{|d_n|}$. These bounds, together with properties of the Weil height, give $H(\tau_n) \le C' f_n^4 |d_n|^2$. Liouville's inequality then yields
\[
|\tau_n - \tau| \;\ge\; \frac{1}{(2H(\tau)H(\tau_n))^2} \;\ge\; \frac{\kappa}{f_n^{8}|d_n|^4}.
\]
Finally, we identify $\operatorname{End}(A_n,\iota_n)$ with the order $R_n \subset K_n$ of discriminant $f_n^{2} d_n$, so $|\operatorname{disc}(P_n)|=f_n^{2}|d_n|$. This proves the theorem.

The paper is organized as follows. Section 2 collects the necessary preliminaries on quaternion algebras, abelian surfaces, the moduli interpretation of Shimura curves, optimal embeddings of imaginary quadratic fields, and the description of CM points as fixed points. Section 3 recalls heights and Liouville's inequality. Section 4 contains the proof of the main theorem, including the explicit representation of the fixed points $\tau,\tau_n$, the convergence argument, the bounds on the coefficients, the height estimate, the application of Liouville's inequality, and the final identification of the discriminant. 

\section{Preliminaries}

\subsection{Shimura curves}

Let $B$ be a quaternion algebra over $\mathbb{Q}$. For each place $v$ of $\mathbb{Q}$, $B_{v}:=B \otimes_{\mathbb{Q}} \mathbb{Q}_v$ is a quaternion algebra over $\mathbb{Q}_v$. 

\begin{definicion}
If $B_v$ is a division algebra, we say that $B$ is \textbf{ramified} at $v$; otherwise, $B$ is \textbf{unramified} at $v$. The \textbf{reduced discriminant} $D$ of $B$ is the product of the distinct primes at which $B$ is ramified. Furthermore, $B$ is said to be \textbf{indefinite} if it is unramified at the infinite place, that is, if $B \otimes_{\mathbb{Q}} \mathbb{R} \cong M_2(\mathbb{R})$.
\end{definicion}

\begin{definicion}
Let $B$ be a quaternion algebra over $\mathbb{Q}$. We write $B = \left(\frac{a,b}{\mathbb{Q}}\right)$ with $a,b \in \mathbb{Q}^{\times}$, which means that $B$ has a $\mathbb{Q}$-basis $\{1,i,j,ij\}$ satisfying $i^2 = a$, $j^2 = b$, and $ij = -ji$.
\end{definicion}

\begin{observacion}
Let \(B = \left(\frac{a,b}{\Q}\right)\) be a quaternion algebra over $\mathbb{Q}$ for some $a, b\in \mathbb{Q}^{\times}$ and let \(D\) denote its reduced discriminant. Then \(B\) is indefinite if and only if \(a>0\) or \(b>0\). Equivalently, this occurs if and only if \(D\) is the product of an even number of distinct primes (see \cite[Exercise 2.4 and Theorem 14.1.3]{Voight}).

Suppose that $a>0$. Fix an injective $\mathbb{Q}$-algebra homomorphism 
\[
\varphi: B \hookrightarrow M_2(\mathbb{R}) 
\]
defined by
	\begin{equation}\label{homoiny1}
	\varphi(x + yi + zj + tij) = \begin{pmatrix}
	x + y\sqrt{a} & z + t\sqrt{a} \\
	b(z - t\sqrt{a}) & x - y\sqrt{a}
	\end{pmatrix}.
	\end{equation}
(see \cite[Remark 2.3.12]{Voight}).
\end{observacion}

\begin{proposicion}\label{propo1}
Let $B$ be a quaternion algebra over $\mathbb{Q}$. Then:
	\begin{enumerate}
		\item $B$ is ramified only at a finite even number of places.
		\item Two quaternion algebras over $\mathbb{Q}$ are isomorphic if and only if they are ramified at the same places.
	\end{enumerate}
\end{proposicion}
\begin{proof}
See \cite[Theorem 14.1.3 and Corollary 14.2.3]{Voight}.
\end{proof}

\begin{proposicion}\label{alg_quat_1}
Let $\Sigma$ be a finite set of places of $\mathbb{Q}$ of even cardinality such that $\infty \notin \Sigma$. Then there exists a quaternion algebra $B$ over $\mathbb{Q}$ that is ramified exactly at the places in $\Sigma$.  
\end{proposicion}

\begin{proof}
See \cite[Proposition 14.2.7]{Voight}.
\end{proof}

\begin{observacion}\label{obs1}
Let $B$ be an indefinite quaternion algebra over $\mathbb{Q}$ with reduced discriminant $D>1$. By Proposition \ref{propo1}, $D$ is a product of an even number of distinct primes. By Proposition \ref{alg_quat_1}, there exists a quaternion algebra $B'$ over $\mathbb{Q}$ with reduced discriminant $D$. Following the construction in the proof of Proposition \ref{alg_quat_1}, there exists a prime $q$ such that $B'= \left(\frac{D,q}{\Q}\right)$. Therefore, by Proposition \ref{propo1}, $B\cong  \left(\frac{D,q}{\Q}\right)$.
\end{observacion}

\begin{definicion}
Let $B$ be a quaternion algebra over $\mathbb{Q}$ and let $D$ denote its reduced discriminant. A \textbf{polarization} on a maximal order $\mathcal{O}$ of $B$ is an element $\mu \in \mathcal{O}$ such that $\mu^2 \in \mathbb{Z}_{<0}$. This polarization is said to be \textbf{principal} if $\mu^2 + D = 0$. In this case, we define the involution
\[
\varsigma: B \longrightarrow B
\]
given by
\[
\varsigma(\alpha) = \mu^{-1} \overline{\alpha} \mu,
\]
where $\overline{\alpha}$ denotes the standard conjugation on $B$. This involution satisfies $\operatorname{Trd}(\varsigma(\alpha)\alpha) > 0$ for all $\alpha \in B\setminus\{0\}$; that is, it is a positive involution.
\end{definicion}

\begin{definicion}
A \textbf{quaternionic multiplication (QM) structure} by $(\mathcal{O}, \mu)$ on a principally polarized abelian surface $A$ is an injective ring homomorphism
\[
\iota: \mathcal{O} \to \operatorname{End}(A)
\]
such that the induced homomorphism
\[
\iota: B \to \operatorname{End}(A) \otimes_{\mathbb{Z}} \mathbb{Q}
\]
makes the following diagram commute:
\[
\begin{tikzcd}[row sep=1.5cm, column sep=2.5cm]
B \arrow[r, "\iota"] \arrow[d, "\varsigma"'] & \operatorname{End}(A) \otimes_{\mathbb{Z}} \mathbb{Q} \arrow[d, "\varrho"] \\
B \arrow[r, "\iota"] & \operatorname{End}(A) \otimes_{\mathbb{Z}} \mathbb{Q}.
\end{tikzcd}
\]
Here, $\varrho$ is the Rosati involution. 

We say that $A$ has quaternionic multiplication (QM) by $(\mathcal{O}, \mu)$ if it is equipped with a QM structure by $(\mathcal{O}, \mu)$.
\end{definicion}

\begin{definicion}
Let $(A, \iota)$ and $(A', \iota')$ be two principally polarized abelian surfaces with QM by $(\mathcal{O}, \mu)$. A \textbf{homomorphism} $(A, \iota) \to (A', \iota')$ is a homomorphism of polarized abelian surfaces $f: A \to A'$ that also respects the QM structures, meaning that the diagram
\[
\begin{tikzcd}[row sep=1.5cm, column sep=2.5cm]
B \arrow[r, "\iota'"] \arrow[dr, "\iota"'] & \operatorname{End}(A') \otimes_{\mathbb{Z}} \mathbb{Q} \arrow[d, "f^*"] \\
& \operatorname{End}(A) \otimes_{\mathbb{Z}} \mathbb{Q}
\end{tikzcd}
\]
commutes. Here, $f^*: \operatorname{End}(A') \otimes_{\mathbb{Z}} \mathbb{Q} \to \operatorname{End}(A)\otimes_{\mathbb{Z}} \mathbb{Q}$ is the map defined by $f^*(g) = f^{-1} \circ g \circ f$ for all $g \in \operatorname{End}(A')\otimes_{\mathbb{Z}} \mathbb{Q}$.
\end{definicion}

\begin{observacion}
Let $(A, \iota)$ be a principally polarized abelian surface with QM by $(\mathcal{O}, \mu)$, then its ring of endomorphisms is $$\operatorname{End}(A, \iota)=\left\lbrace f \in \operatorname{End}(A): f \circ \iota(\alpha)=\iota(\alpha)\circ f \text{  for all } \alpha\in \mathcal{O} \right\rbrace. $$
\end{observacion}

\subsection{Moduli interpretation}

Let \(B = \left(\frac{a,b}{\Q}\right)\) be an indefinite quaternion algebra of reduced discriminant $D>1$, where $a, b \in \mathbb{Q}^{\times}$ and \(a>0\). Fix a maximal order $\mathcal{O}$ in $B$, and define 
\[
\cO_+^{\times} = \{ \gamma \in \cO^{\times} : \Nrd(\gamma) = 1 \}.
\]
Under the fixed homomorphism $\varphi: B \hookrightarrow M_2(\mathbb{R})$ defined in \eqref{homoiny1}, the group \(\Gamma(\mathcal{O}, 1) = \varphi(\mathcal{O}_+^{\times}) / \{\pm 1\}\) is a Fuchsian group acting on the upper half-plane \(\mathcal{H}\). The quotient space
\[
X(D, 1) = \Gamma(\mathcal{O}, 1) \backslash \mathcal{H}
\]
is compact (see \cite[Lemma 38.1.2 and Proposition 38.1.3]{Voight}). The quotient $X(D, 1)$ is called the Shimura curve associated with the subgroup \(\Gamma(\mathcal{O}, 1)\), which is known to be a compact Riemann surface. Consequently, \(\Gamma(\mathcal{O}, 1)\) is of the first kind (see \cite[Section 1.5]{Shimura}). 

The homomorphism \(\varphi: B \to M_2(\R)\) defined in \eqref{homoiny1} naturally extends to \(\varphi: B \to M_2(\C)\). For any \(\tau \in \mathcal{H}\), we define
\[
\Lambda_{\tau} := \varphi(\mathcal{O}) \begin{pmatrix} \tau \\ 1 \end{pmatrix} \subset \C^2.
\]
Since \(\operatorname{rank}_{\Z}(\mathcal{O}) = 4\), \(\Lambda_{\tau}\) is a lattice in \(\C^2\). Let \(A_{\tau} := \C^2 / \Lambda_{\tau}\) denote the complex torus associated with \(\Lambda_{\tau}\). 

The inclusion \(\varphi(\mathcal{O})\Lambda_{\tau} \subset \Lambda_{\tau}\) ensures that the map \(\varphi\) induces an injective ring homomorphism
\[
\iota_{\tau}: \mathcal{O} \longrightarrow \End(A_{\tau}), \quad \alpha \longmapsto \iota_{\tau}(\alpha),
\]
where the action is given by \(\iota_{\tau}(\alpha)(z + \Lambda_{\tau}) := \varphi(\alpha)z + \Lambda_{\tau}\). 

Furthermore, consider the form \(E_{\tau}: \Lambda_{\tau} \times \Lambda_{\tau} \to \Z\) defined by
\[
E_{\tau}\left( x \begin{pmatrix} \tau \\ 1 \end{pmatrix}, y \begin{pmatrix} \tau \\ 1 \end{pmatrix} \right) = \frac{1}{D} \Trd(\varphi(\mu) x \overline{y})
\]
for \(x, y \in \varphi(\mathcal{O})\). The relation \(\Trd(\mu \mathcal{O}) \subset D\Z\) guarantees that \(E_{\tau}\) indeed takes values in \(\Z\) (see \cite[Lemma 43.6.7]{Voight}). 

\begin{lema}\label{lem:properties_E_tau}
The form \(E = E_{\tau}\) or $E= -E_{\tau}$ is a non-degenerate Riemann form that induces a principal polarization on \(A_{\tau}\). Moreover, the homomorphism \(\iota_{\tau}: \mathcal{O} \to \End(A_{\tau})\otimes_{\mathbb{Z}} \mathbb{Q}\) respects involutions, and \(E\) is the unique principal polarization on \(A_{\tau}\) compatible with \(\iota_{\tau}\).
\end{lema}

\begin{proof}
See \cite[Lemmas 43.6.16, 43.6.22, and 43.6.23]{Voight}.
\end{proof}

\begin{proposicion}\label{prop:moduli_isomorphism}
Every principally polarized complex abelian surface with QM by \((\mathcal{O}, \mu)\) is isomorphic to \((A_{\tau}, \iota_{\tau})\) for some \(\tau \in \mathcal{H}\).
\end{proposicion}

\begin{proof}
See \cite[Proposition 43.6.28]{Voight}.
\end{proof}

Using these results, we establish the bijective correspondence between points on the Shimura curve and isomorphism classes of principally polarized complex abelian surfaces with quaternionic multiplication.

\begin{teorema}\label{interpretacion_moduli}
Let \(\Gamma(\mathcal{O}, 1) = \varphi(\mathcal{O}_+^{\times}) / \{\pm 1\}\). The map
\[
\Theta : \Gamma(\mathcal{O}, 1) \backslash \mathcal{H} \longrightarrow \left\{ \begin{array}{c}
\text{Principally polarized complex abelian} \\ 
\text{surfaces } (A, \iota) \text{ with QM by } (\mathcal{O}, \mu) \\ 
\text{up to isomorphism}
\end{array} \right\}
\]
given by \(\Gamma(\mathcal{O}, 1)\cdot\tau \longmapsto [(A_{\tau}, \iota_{\tau})]\) is a bijection.
\end{teorema}

\begin{proof}
See \cite[Theorem 43.6.14]{Voight}.
\end{proof}

\subsection{Optimal homomorphisms and CM points}

Let \(B\) be a quaternion algebra over \(\mathbb{Q}\) of reduced discriminant $D>1$, and let \(\mathcal{O}\) be a maximal order in \(B\). Let \(K\) be a quadratic field and \(R\) an order in \(K\). An injective ring homomorphism \(\phi: R \to \mathcal{O}\) induces, by extension of scalars, an injective \(\mathbb{Q}\)-algebra homomorphism \(\phi: K \to B\).

\begin{definicion}
An injective homomorphism \(\phi: R \to \mathcal{O}\) is said to be \textbf{optimal} if \(\phi(K) \cap \mathcal{O} = \phi(R)\).
\end{definicion}

\begin{observacion}
The optimality condition is equivalent to \(R = \phi^{-1}(\phi(K) \cap \mathcal{O})\). Moreover, if \(R\) is the ring of integers of \(K\), then every injective homomorphism \(\phi: R \to \mathcal{O}\) is optimal.
\end{observacion}

\begin{definicion}
Let $a\in \mathbb{Z}$ and let $p$ be a prime number. The \textbf{Kronecker symbol} $\left(\dfrac{a}{p}\right)$ is defined as follows: for an odd prime $p$,
$$ \left( \dfrac{a}{p}\right) = \left\{ \begin{array}{cl}
             1  & \text{if $a$ is a nonzero square modulo $p$,} \\
             0  & \text{if $p \mid a$,} \\
             -1 & \text{if $a$ is not a square modulo $p$,}
             \end{array}
   \right.$$
and for $p=2$,  
$$ \left( \dfrac{a}{2}\right) = \left\{ \begin{array}{cl}
             1  & \text{if $a\equiv \pm 1 \pmod{8}$,} \\
             0 & \text{if $2 \mid a$,} \\
             -1 & \text{if $a\equiv \pm 5 \pmod{8}$.}
             \end{array}
   \right.$$ 
\end{definicion}

\begin{lema}\label{lema_2}
Let $d$ be the discriminant of $K$. Then there exists a homomorphism of $K$ into $B$ if and only if for every prime $p$ such that $p \mid D$, we have 
$$\left( \dfrac{d}{p}\right)\neq 1.$$  
\end{lema}

\begin{proof}
See \cite[Lemma 4.3]{Alsina-Bayer}.
\end{proof}

\begin{proposicion}
Suppose that there exists a homomorphism \(\phi: K \to B\). Let \(\varphi: B \to M_2(\mathbb{R})\) be the homomorphism defined in \eqref{homoiny1}. Then, all transformations $\gamma \neq id$ in \(\varphi(\phi(K^{\times}))\) have the same fixed point. In particular, if $K$ is an imaginary quadratic field, then they have a unique fixed point in $\mathcal{H}$.
\end{proposicion}

\begin{proof}
See \cite[Corollary 6.2]{Alsina-Bayer}.
\end{proof}

\begin{definicion}
Let $X(D,1) = \Gamma(\mathcal{O},1)\backslash \mathcal{H}$ be the Shimura curve associated with $\Gamma(\mathcal{O},1)$. Let $K$ be an imaginary quadratic field for which there exists a homomorphism of $K$ into $B$, and let $R \subset K$ be an order. A point $\Gamma(\mathcal{O},1)\cdot \tau \in X(D,1)$ is called a \textbf{CM point by $R$} if there exists an optimal homomorphism $\phi: R \hookrightarrow \mathcal{O}$ such that $\tau$ is the unique fixed point of \(\varphi(v)\) in $\mathcal{H}$ for every \(v \in \phi(K^{\times})\).
\end{definicion}

\begin{lema}\label{lema_end}
Let $P=[(A, \iota)]\in X(D,1)$ be a CM point by $R$, and let $\tau \in \mathcal{H}$ be the point corresponding to $P$. Let $\phi: R\hookrightarrow \mathcal{O}$ be an optimal homomorphism such that $\tau$ is the unique fixed point of $\varphi(v)$ in $\mathcal{H}$ for every $v\in \phi(K^{\times})$. Then $\operatorname{End}(A, \iota) \cong R$.
\end{lema}

\begin{proof}
By Theorem \ref{interpretacion_moduli}, $A \cong \mathbb{C}^2/\Lambda_{\tau}$ with $\Lambda_{\tau} = \varphi(\mathcal{O}) \begin{pmatrix} \tau \\ 1 \end{pmatrix}$, where $\varphi(\mathcal{O})\Lambda_{\tau} \subset \Lambda_{\tau}$ and $$\iota(\alpha)(z+ \Lambda_{\tau}) = \varphi(\alpha)z +\Lambda_{\tau}$$ for all $z \in \mathbb{C}^2$ and $\alpha \in \mathcal{O}$. Let $\theta \in \operatorname{End}(A, \iota)$, then $\theta \in \operatorname{End}(A)$ such that $\theta \circ \iota(\alpha) = \iota(\alpha) \circ \theta$ for all $\alpha \in \mathcal{O}$. Thus, there exists a matrix $M_{\theta} \in M_{2}(\mathbb{C})$ such that $M_{\theta}\Lambda_{\tau} \subset \Lambda_{\tau}$ and $\theta(z+\Lambda_{\tau}) = M_{\theta}z +\Lambda_{\tau}$ for all $z \in \mathbb{C}^2$. 

Furthermore, $$(\theta \circ \iota(\alpha))(z+\Lambda_{\tau}) = \theta(\varphi(\alpha)z +\Lambda_{\tau}) = M_{\theta}\varphi(\alpha)z +\Lambda_{\tau},$$ and $$(\iota(\alpha) \circ \theta)(z+\Lambda_{\tau}) = \iota(\alpha)(M_{\theta}z +\Lambda_{\tau}) = \varphi(\alpha)M_{\theta}z +\Lambda_{\tau},$$
for all $z\in \mathbb{C}^2$ and $\alpha\in \mathcal{O}$.
This implies that $M_{\theta}\varphi(\alpha)z +\Lambda_{\tau} = \varphi(\alpha)M_{\theta}z +\Lambda_{\tau}$ for all $z \in \mathbb{C}^2$. Defining $N := M_{\theta}\varphi(\alpha) - \varphi(\alpha)M_{\theta} \in M_{2}(\mathbb{C})$, we obtain that $Nz \in \Lambda_{\tau}$ for all $z \in \mathbb{C}^2$, which means $\operatorname{Im}(N) \subset \Lambda_{\tau}$. Since $\Lambda_{\tau}$ is discrete, we must have $\operatorname{Im}(N) = \{0\}$, so $N = 0$. Therefore, $M_{\theta}\varphi(\alpha) = \varphi(\alpha)M_{\theta}$ for all $\alpha \in \mathcal{O}$.

Since the centralizer $C_{M_2(\mathbb{C})}(M_2(\mathbb{R})) = \mathbb{C}I_2$ and $\varphi(\mathcal{O})$ contains an $\mathbb{R}$-basis of $M_2(\mathbb{R})$, we deduce that $M_{\theta} = \lambda_{\theta}I_2$ for some $\lambda_{\theta} \in \mathbb{C}$. 

Given that $M_{\theta}\Lambda_{\tau} \subset \Lambda_{\tau}$, we have $\lambda_{\theta}\begin{pmatrix} \tau \\ 1 \end{pmatrix} \in \Lambda_{\tau}$. Thus, there exists $\alpha \in \mathcal{O}$ such that $\lambda_{\theta}\begin{pmatrix} \tau \\ 1 \end{pmatrix} = \varphi(\alpha)\begin{pmatrix} \tau \\ 1 \end{pmatrix}$, meaning $\lambda_{\theta}$ is an eigenvalue of $\varphi(\alpha)$ with corresponding eigenvector $\begin{pmatrix} \tau \\ 1 \end{pmatrix}$. Let $\omega \in K^{\times}$, then $\varphi(\phi(\omega)) \cdot \tau = \tau$, and therefore $\varphi(\phi(\omega))\begin{pmatrix} \tau \\ 1 \end{pmatrix} = \mu \begin{pmatrix} \tau \\ 1 \end{pmatrix}$ for some $\mu \in \mathbb{C}$. The matrices $\varphi(\alpha)$ and $\varphi(\phi(\omega))$ are two real matrices that share the same basis of eigenvectors (since $\tau \notin \mathbb{R}$, they share both $\begin{pmatrix} \tau \\ 1 \end{pmatrix}$ and its complex conjugate), so they commute. This implies $\alpha \phi(\omega) = \phi(\omega)\alpha$. Thus, $\alpha \in C_B(\phi(K)) = \phi(K)$, which means $\alpha \in \phi(K) \cap \mathcal{O} = \phi(R)$. Therefore, $\alpha = \phi(\beta)$ for some $\beta \in R$. Hence, $\varphi(\phi(\beta)) \begin{pmatrix} \tau \\ 1 \end{pmatrix} = \lambda_{\theta}\begin{pmatrix} \tau \\ 1 \end{pmatrix},$ which means $\lambda_{\theta}$ is an eigenvalue of $\varphi(\phi(\beta))$. Note that the characteristic polynomial of $\varphi(\phi(\beta))$ is the same as the minimal polynomial of $\beta$ over $\mathbb{Q}$. This implies that $\lambda_{\theta} = \beta$ or $\lambda_{\theta} = \overline{\beta}$, so $\lambda_{\theta} \in R$.

Conversely, let $r \in R$. Since $\varphi(\phi(r)) \cdot \tau = \tau$, we have $\varphi(\phi(r))\begin{pmatrix} \tau \\ 1 \end{pmatrix} = \mu \begin{pmatrix} \tau \\ 1 \end{pmatrix},$ for some $\mu \in \mathbb{C}$, allowing us to conclude that $\mu = r$ or $\mu = \overline{r}$. Define $M_{\mu} := \mu I_2 \in M_2(\mathbb{C})$. It follows that $M_{\mu}\Lambda_{\tau} \subset \Lambda_{\tau}$, which induces an endomorphism $\theta_{\mu}$ of $A$ such that $\theta_{\mu}(z+\Lambda_{\tau}) = M_{\mu}z+\Lambda_{\tau}$ for all $z \in \mathbb{C}^2$. 

Furthermore, for any $\alpha \in \mathcal{O}$, we have $$\theta_{\mu} \circ \iota(\alpha)(z+\Lambda_{\tau}) = \theta_{\mu}(\varphi(\alpha)z +\Lambda_{\tau}) = M_{\mu}\varphi(\alpha)z +\Lambda_{\tau}$$ and $$\iota(\alpha) \circ \theta_{\mu}(z+\Lambda_{\tau}) = \iota(\alpha)(M_{\mu}z +\Lambda_{\tau}) = \varphi(\alpha)M_{\mu}z +\Lambda_{\tau}$$ 
for all $z\in \mathbb{C}^2$. Since $M_{\mu}$ is a scalar matrix, it commutes with $\varphi(\alpha)$ (i.e., $M_{\mu} \varphi(\alpha) = \varphi(\alpha)M_{\mu}$). This equality holds for all $\alpha \in \mathcal{O}$, proving that $\theta_{\mu} \in \operatorname{End}(A, \iota)$.

This proves that $\operatorname{End}(A, \iota) \cong R$.
\end{proof}

\begin{proposicion}\label{prop_tau}
Suppose that $B = \left(\frac{a,b}{\mathbb{Q}}\right)$ is an indefinite quaternion algebra, with $a,b\in\mathbb{Q}^{\times}$ and $a>0$. Let $d$ be the discriminant of $K$, and let $\tau_R=f\frac{d+\sqrt{d}}{2}$ be a generator of $R$, where $f$ is the conductor of $R$. Let $\Gamma(\mathcal{O},1)\cdot \tau \in X(D,1)$ be a CM point by $R$, and let $\phi: R \hookrightarrow \mathcal{O}$ be an optimal homomorphism such that $\tau$ is the unique fixed point of $\varphi(v)$ in $\mathcal{H}$ for every $v\in \phi(K^{\times})$. Write $\phi(\sqrt{d}) = y\,i + z\,j + t\,ij$ with $y,z,t\in\mathbb{Q}$. Then
\[
\tau = \frac{y\sqrt{a} + \varepsilon \sqrt{d}}{b\,(z - t\sqrt{a})},
\]
where $\varepsilon= \pm 1$ is chosen so that \(\tau \in \mathcal{H}\). Additionally, we have $\phi(\sqrt{d})\in \frac{1}{f} \mathcal{O}$.
\end{proposicion}

\begin{proof}
Since $\tau$ is the unique fixed point of $\varphi(v)$ for every $v\in \phi(K^{\times})$, then taking $v=\phi(\tau_R)$ we obtain $((\varphi \circ \phi) (\tau_R))\cdot \tau =\tau$. The homomorphism $\phi$ extends to $K$, hence
$$\phi(\tau_R) = \frac{fd}{2} + \frac{fy}{2}i + \frac{fz}{2}j + \frac{ft}{2}ij.$$
Applying $\varphi$ yields
$$\varphi(\phi(\tau_R)) = \begin{pmatrix}
\frac{fd}{2} + \frac{fy}{2}\sqrt{a} & \frac{fz}{2} + \frac{ft}{2}\sqrt{a} \\[2mm]
b\left(\frac{fz}{2} - \frac{ft}{2}\sqrt{a}\right) & \frac{fd}{2} - \frac{fy}{2}\sqrt{a}
\end{pmatrix}.$$
The fixed-point condition on $\tau$ leads to 
$$b(z-t\sqrt{a})\tau^2 -2y\sqrt{a}\tau -z-t\sqrt{a}=0.$$
As $\phi(\sqrt{d})$ is a root of the polynomial $x^2 -d \in \mathbb{Q}[x]$, we have $ay^2 +bz^2 -abt^2=d$, from which we obtain
$$\tau=\frac{y\sqrt{a}+\varepsilon\sqrt{d}}{b(z-t\sqrt{a})}.$$
On the other hand, since $\tau_R=f\frac{d+\sqrt{d}}{2} \in R$, we have $\sqrt{d}=\frac{1}{f}(2 \tau_R -fd)$. Applying $\phi$, we obtain $\phi(\sqrt{d})=\frac{1}{f}\phi(2\tau_R -fd) \in \frac{1}{f}\mathcal{O}$.
\end{proof}

\begin{observacion}\label{obs_order}
Let $\left\lbrace e_1, e_2, e_3, e_4 \right\rbrace$ be a $\mathbb{Z}$-basis of $\mathcal{O}$. Writing each $e_l$ in terms of the basis $\left\lbrace 1, i, j, ij \right\rbrace $, we have $e_l=q_{l1}\cdot 1+q_{l2}\cdot i+q_{l3}\cdot j+q_{l4}\cdot ij$, where $q_{lr}\in \mathbb{Q}$ for $1\leq l, r \leq 4$. Let $N$ be the least common multiple of all denominators of the $q_{lr}$. Then $Nq_{lr} \in \mathbb{Z}$ for $1\leq l, r \leq 4$, and hence $Ne_l \in \mathbb{Z}\cdot 1 +\mathbb{Z}\cdot i+\mathbb{Z}\cdot j+\mathbb{Z}\cdot ij$ for all $1\leq l \leq 4$. Therefore, $\mathcal{O}\subset \frac{1}{N}\left(\mathbb{Z}\cdot 1 +\mathbb{Z}\cdot i+\mathbb{Z}\cdot j+\mathbb{Z}\cdot ij \right)$, where $N$ depends only on $\mathcal{O}$.
\end{observacion}

\section{Heights and Liouville's inequality}

We define $M_{\mathbb{Q}} = \left\lbrace |\cdot|_p : \text{$p$ is prime or $p=\infty$} \right\rbrace$, normalized as follows. If $p=\infty$, then $|\cdot|_p$ is the ordinary absolute value on $\mathbb{Q}$, and, if $p$ is prime, then the absolute value is the $p$-adic absolute value on $\mathbb{Q}$, with $|p|_{p}=1/p$. For a number field $L$, let $M_L$ denote the set of absolute values on $L$ whose restriction to $\mathbb{Q}$ belongs to $M_{\mathbb{Q}}$.

Let $X\in \mathbb{P}^n (\overline{\mathbb{Q}})$ be a point given by $X=[x_0: \cdots :x_n]$, where $x_0, \ldots, x_n \in L$. We define the absolute logarithmic height of $X$ as
$$h(X) := \sum_{v\in M_L} \max_j \log |x_j|_v.$$ 

The multiplicative height of $X$ is defined by
$$H(X) := \exp(h(X)) = \prod_{v\in M_L} \max_j |x_j|_v.$$
We introduce the standard notation $\log^+(x) = \max\left\lbrace 0, \log(x) \right\rbrace$ for all $x>0$, and set $\log^+(0)=0$. 

Now, let $X=(x_1, \ldots, x_n)\in \mathbb{A}^n(\overline{\mathbb{Q}})$. Via the standard embedding into $\mathbb{P}^n(\overline{\mathbb{Q}})$ given by
$$X=(x_1, \ldots, x_n) \mapsto [1: x_1: \cdots :x_n],$$
we define the height of $X$ as the height of its image; that is,
$$h(X) = h([1: x_1: \cdots :x_n]).$$

Thus, for an algebraic number $\alpha$, the absolute logarithmic height is given by
$$h(\alpha) = \sum_{v\in M_L} \log^+|\alpha|_v,$$ 
where $L$ is a number field containing $\alpha$. Consequently, the multiplicative height of $\alpha$ is
$$H(\alpha) = \prod_{v\in M_L} \max\{1, |\alpha|_v\}.$$

\begin{proposicion}\label{prop_height}
Let $\alpha, \beta \in \overline{\mathbb{Q}}$ be two algebraic numbers. Then:
	\begin{enumerate}
    	\item $H(\alpha\beta) \le H(\alpha)H(\beta)$.
    	\item $H(\alpha+\beta) \le 2 H(\alpha)H(\beta)$.
    	\item $H(\alpha^{-1}) = H(\alpha)$ for $\alpha \neq 0$.
	\end{enumerate}
\end{proposicion}

\begin{proof}
See \cite[Lemma 1.9.2]{Evertse-Gyoy}.
\end{proof}

\begin{teorema}[Liouville's inequality]
Let $L$ and $K$ be number fields such that $L$ is a finite extension of $K$ of degree $r$. For any distinct elements $\alpha \in L$ and $\beta \in K$, we have that  
	\begin{equation}\label{eq:liouville}
	|\alpha -\beta| \geq \frac{1}{(2H(\alpha) H(\beta))^r}.
	\end{equation}
\end{teorema}

\begin{proof}
See \cite[Theorem 1.5.21]{Bombieri}.
\end{proof}

\section{Main Result}

In this section, we will prove the Theorem \ref{mainthm}. By Theorem \ref{interpretacion_moduli}, there exist $\tau, \tau_n \in \mathcal{H}$ that correspond to $P$ and $P_n$, respectively.

\begin{lema}\label{lema_3}
Suppose that the sequence $\{P_n=[(A_n, \iota_n)]\}_{n\in \mathbb{N}}$ converges to $P$. Then there exists a sequence $\{\tau'_n\}_{n\in \mathbb{N}} \subset \mathcal{H}$, with $\tau'_n \in \Gamma(\mathcal{O}, 1) \cdot \tau_n$ for all $n$, such that $\tau'_n \to \tau$ in $\mathcal{H}$.
\end{lema}
\begin{proof}
For each \( n \in \mathbb{N} \), we consider the orbit \( \Gamma(\mathcal{O}, 1) \cdot \tau_n = \{ \gamma \tau_n \mid \gamma \in \Gamma(\mathcal{O}, 1) \} \). Since \( P_n \to P \) in \( X(D,1) \), we have 
\[
\sigma_n := d_{X(D, 1)}(P_n, P) = \inf_{\gamma \in \Gamma(\mathcal{O}, 1)} d_{\mathcal{H}}(\gamma \tau_n, \tau) 
\to 0 \quad \text{as } n \to \infty.
\]
We choose \( \tau'_n \in \Gamma(\mathcal{O}, 1) \cdot \tau_n \) such that
\[
d_{\mathcal{H}}(\tau'_n, \tau) = \min_{\gamma \in \Gamma(\mathcal{O}, 1)} d_{\mathcal{H}}(\gamma \tau_n, \tau).
\]
We first show that the minimum exists. Since \( \sigma_n \to 0 \), it follows that \( \sigma_n \) is bounded, say \( \sigma_n \leq L \) for all \( n \). Let us consider the closed hyperbolic ball
\[
B_{\mathcal{H}}[\tau, L + 1] = \{ z \in \mathcal{H} \mid d_{\mathcal{H}}(z, \tau) \leq L + 1 \},
\]
which is compact in \( \mathcal{H} \). Given that \( \Gamma(\mathcal{O}, 1) \) acts properly discontinuously on \( \mathcal{H} \), the intersection
\[
(\Gamma(\mathcal{O}, 1) \cdot \tau_n) \cap B_{\mathcal{H}}[\tau, L+1]
\]
is finite. Consequently, the infimum is attained, which means the minimum exists, and \( \tau'_n \) is well-defined.

By the choice of \(\tau'_n \) and the definition of the metric in \( X(D,1) \), we have
\[
d_{\mathcal{H}}(\tau'_n, \tau) = \min_{\gamma \in \Gamma(\mathcal{O}, 1)} d_{\mathcal{H}}(\gamma \tau_n, \tau) = d_{X(D, 1)}(P_n, P) = \sigma_n \to 0.
\]
Therefore, \(\tau'_n \to \tau \) in \( \mathcal{H} \).
\end{proof} 

By Lemma \ref{lema_3}, we can choose $\tau_n \in \mathcal{H}$ corresponding to $P_n$ such that $\tau_n\to\tau$.

By Remark \ref{obs1}, $B=\left(\frac{D, q}{\mathbb{Q}} \right)$ for some prime $q$. Let $\mathcal{O}$ be a maximal order in $B$. We fix the injective $\mathbb{Q}$-algebra homomorphism $\varphi: B \hookrightarrow M_2(\mathbb{R})$ given by 
\[
\varphi(x_0 + x_1 i + x_2 j + x_3 ij) = 
\begin{pmatrix}
x_0 + x_1\sqrt{D} & x_2 + x_3\sqrt{D} \\
q(x_2 - x_3\sqrt{D}) & x_0 - x_1\sqrt{D}
\end{pmatrix},
\]
as in \eqref{homoiny1}.

Let $K=\mathbb{Q}(\sqrt{d})$ be an imaginary quadratic field, where $d<0$ is the discriminant of $K$. We assume that 
$\left(\frac{d}{p}\right) \neq 1,$ for every prime $p$ such that $p \mid D$. It follows from Lemma \ref{lema_2} that there exists a homomorphism from $K$ into $B$.

Let $R$ be an order of $K$ given by 
$$R = \mathbb{Z} + \mathbb{Z}f\frac{d + \sqrt{d}}{2},$$
where $f$ is the conductor of $R$. The complex number
$$\tau_R =f \frac{d + \sqrt{d}}{2} \in \mathcal{H}$$
is a generator of $R$. On the other hand, let $\{d_n\}_{n\in\mathbb{N}}$ be a sequence of negative discriminants such that $|d_n| \to \infty$ as $n \to \infty$, and suppose that $\left(\frac{d_n}{p}\right) \neq 1$, for every prime $p$ such that $p \mid D$. For each $n$, let $K_n = \mathbb{Q}(\sqrt{d_n})$ be the corresponding imaginary quadratic field. By Lemma \ref{lema_2}, there exists a homomorphism from $K_n$ into $B$.

Let $R_n$ be an order of $K_n$ given by
$$R_n = \mathbb{Z} + \mathbb{Z}f_n\frac{d_n + \sqrt{d_n}}{2},$$
where $f_n$ is the conductor of $R_n$. For each $n$, the complex number
$$\tau_{R_n} = f_n\frac{d_n + \sqrt{d_n}}{2} \in \mathcal{H}$$
is a generator of $R_n$.

Since $P=[(A, \iota)] \in X(D, 1)$ is a CM point by $R$, there exists an optimal homomorphism $\phi: R \hookrightarrow \mathcal{O}$ such that $\tau$ is the unique fixed point of $\varphi(v)$ for every $v\in \phi(K^{\times})$. By Proposition \ref{prop_tau} and Remark \ref{obs_order}, we can write
$$\phi(\sqrt{d}) = \frac{y}{Nf}\,i + \frac{z}{Nf}\,j + \frac{t}{Nf}\,ij,$$
with $y,z,t\in\mathbb{Z}$. Since $\phi(\sqrt{d})$ is a root of the polynomial $x^2 -d \in \mathbb{Q}[x]$, it satisfies
	\begin{equation}\label{equa:Dq_1}
	D\frac{y^2}{N^2 f^2} +q\frac{z^2}{N^2 f^2}-Dq\frac{t^2}{N^2 f^2}=d.
	\end{equation} 

Given that $\left\lbrace P_n=[(A_n, \iota_n)]\right\rbrace_{n\in \mathbb{N}}$ is a sequence in $X(D, 1)$ of CM points by $R_n$, for each $n$, there exists an optimal homomorphism $\phi_n: R_n \hookrightarrow \mathcal{O}$ such that $\tau_n$ is the unique fixed point of $\varphi(v)$ for every $v\in \phi_n(K_n^{\times})$. By Proposition \ref{prop_tau} and Remark \ref{obs_order}, we can write
$$\phi_n(\sqrt{d_n}) = \frac{y_n}{N f_n}\,i + \frac{z_n}{N f_n}\,j + \frac{t_n}{N f_n}\,ij,$$
with $y_n, z_n, t_n\in\mathbb{Z}$. Since $\phi_n(\sqrt{d_n})$ is a root of the polynomial $x^2 -d_n \in \mathbb{Q}[x]$, it satisfies
	\begin{equation}\label{equa:Dq_2}
	D\frac{y_n^2}{N^2 f_n^2} +q\frac{z_n^2}{N^2 f_n^2}-Dq\frac{t_n^2}{N^2 f_n^2}=d_n.
	\end{equation} 

Applying Proposition \ref{prop_tau} to $\tau$ and $\tau_n$, we obtain that 
$$\tau=\frac{y \sqrt{D}+\varepsilon N f\sqrt{d}}{q(z-t\sqrt{D})} \qquad \text{and} \qquad \tau_n=\frac{y_n \sqrt{D}+\varepsilon_n N f_n\sqrt{d_n}}{q(z_n -t_n\sqrt{D})},$$
where $\varepsilon=\pm 1$ and $\varepsilon_n=\pm 1$ are chosen so that $\tau, \tau_n \in \mathcal{H}$.

Using the relation (\ref{equa:Dq_2}) and the convergence $\tau_n\to\tau$, we will obtain upper bounds for $|y_n|,|z_n|,|t_n|$ in terms of $|d_n|$.

\begin{lema}\label{lema_1}
There exists a constant $C>0$ such that for all sufficiently large $n$, we have
$$|y_n|,\ |z_n|,\ |t_n| \le C f_n\sqrt{|d_n|}.$$
\end{lema}

\begin{proof}
As $d_n < 0$, we have that
$$\tau_n = \frac{y_n\sqrt{D} + \varepsilon_n N f_n\sqrt{d_n}}{q(z_n - t_n\sqrt{D})} = \frac{y_n\sqrt{D}}{q(z_n - t_n\sqrt{D})} +\sqrt{-1}\frac{\varepsilon_n N f_n\sqrt{|d_n|}}{q(z_n - t_n\sqrt{D})}.$$
It then follows that
$$\lim_{n \to \infty} \frac{y_n\sqrt{D}}{q(z_n - t_n\sqrt{D})}=\frac{y\sqrt{D}}{q(z-t\sqrt{D})} \qquad \text{and} \qquad \lim_{n \to \infty} \frac{\varepsilon_n N f_n\sqrt{|d_n|}}{q(z_n - t_n\sqrt{D})}=\frac{\varepsilon N f\sqrt{|d|}}{q(z-t\sqrt{D})}.$$ 
We set 
$$U=\frac{y\sqrt{D}}{q(z-t\sqrt{D})} \qquad \text{and} \qquad V=\frac{\varepsilon N f\sqrt{|d|}}{q(z-t\sqrt{D})}>0.$$  

Since $V>0$, given $\delta>0$, there exists $M\in\mathbb{N}$ such that for all $n>M$, we have
$$\left| \frac{\varepsilon_n N f_n\sqrt{|d_n|}}{q(z_n-t_n\sqrt{D})} -V\right|<\delta \quad \text{ and } \quad \left| \frac{q(z_n-t_n\sqrt{D})}{\varepsilon_n N f_n\sqrt{|d_n|}} -\frac{1}{V}\right|<\delta,$$
hence,
	\begin{equation}\label{eq1} 
	\sqrt{|d_n|}<\frac{q}{|N|f_n}(\delta+V)|z_n-t_n \sqrt{D}| \quad \text{ and } \quad |z_n-t_n \sqrt{D}|<\frac{1}{q}			\left(\delta+\frac{1}{V} \right)|N|f_n \sqrt{|d_n|}.
	\end{equation}

Similarly, we may assume that for all $n>M$, we have 
$$\left|\frac{y_n \sqrt{D}}{q(z_n-t_n\sqrt{D})} -U \right|<\delta,$$
therefore,
$$\left|\frac{y_n\sqrt{D}}{q(z_n-t_n\sqrt{D})} \right|<\delta+|U|, $$
hence, 
$$|y_n|<\frac{q(\delta+|U|)}{\sqrt{D}}|z_n-t_n\sqrt{D}|.$$
Using the relation (\ref{eq1}), we conclude that 
	\begin{equation}\label{eq2}
	|y_n|<\frac{(\delta+|U|)}{\sqrt{D}}\left(\delta+\frac{1}{V}\right)|N| f_n\sqrt{|d_n|}.
	\end{equation}

On the other hand, we use the relation \eqref{equa:Dq_2} and the identity $z_n=(z_n-t_n\sqrt{D}) +t_n\sqrt{D}$. For $n>M$,
$$D \frac{y_n^2}{N^2 f_n^2} + q\frac{((z_n - t_n\sqrt{D})+ t_n\sqrt{D})^2}{N^2 f_n^2} -Dq\frac{t_n^2}{N^2 f_n^2} = -|d_n|,$$
thus 
$$D\frac{y_n^2}{N^2 f_n^2} +q\frac{(z_n -t_n \sqrt{D})^2}{N^2 f_n^2} +\frac{2q\sqrt{D}(z_n-t_n\sqrt{D})t_n}{N^2 f_n^2}=-|d_n|,$$
from which we obtain
$$\frac{2q\sqrt{D}(z_n-t_n\sqrt{D})t_n}{N^2 f_n^2}=-|d_n|-D\frac{y_n^2}{N^2 f_n^2} -\frac{q(z_n-t_n\sqrt{D})^2}{N^2 f_n^2}.$$

From the relations (\ref{eq1}) and (\ref{eq2}), it follows that
	\begin{equation}\label{eq3}
	\frac{2q\sqrt{D}|z_n -t_n\sqrt{D}||t_n|}{N^2 f_n^2}<\left(1+(\delta+|U|)^2 \left(\delta+\frac{1}{V} \right)^2 +\frac{1}{q}	\left(\delta +\frac{1}{V} \right)^2 \right) |d_n|.
	\end{equation}
From the relation (\ref{eq1}), we have that
$$\frac{|d_n|}{|z_n-t_n\sqrt{D}|}<\frac{q}{|N|f_n}(\delta+V)\sqrt{|d_n|},$$
and replacing in (\ref{eq3}), we obtain
	\begin{equation}\label{eq4}
	|t_n|<\frac{(\delta+V)}{2\sqrt{D}}\left(1+(\delta+|U|)^2 \left(\delta+\frac{1}{V} \right)^2 +\frac{1}{q}\left(\delta +	\frac{1}{V} \right)^2 \right)|N| f_n\sqrt{|d_n|}.
	\end{equation}
On the other hand, we have the triangle inequality
$$|z_n|\leq |z_n-t_n\sqrt{D}|+\sqrt{D}|t_n|.$$
Using the relations (\ref{eq1}) and (\ref{eq4}), we obtain
$$|z_n|<\frac{1}{q}\left(\delta +\frac{1}{V}\right)|N|f_n\sqrt{|d_n|} +\frac{(\delta+V)}{2}\left(1+(\delta+|U|)^2 \left(\delta+\frac{1}{V} \right)^2 +\frac{1}{q}\left(\delta +\frac{1}{V} \right)^2 \right)|N| f_n\sqrt{|d_n|}.$$
Taking the maximum of the respective constants yields the result.
\end{proof}

By Lemma \ref{lema_end}, for each $n$, $\operatorname{disc}(P_n)=\operatorname{disc}(\operatorname{End}(A_n, \iota_n))=\operatorname{disc}(R_n)=f_n^2 d_n$. 

\begin{proof}[\textbf{Proof of Theorem \ref{mainthm}}]
Since $d_n<0$ for each $n$, we have $H(\sqrt{d_n}) = \sqrt{|d_n|}$.

Using Proposition \ref{prop_height}, we obtain:
$$
\begin{aligned}
H(\tau_n) &= H\!\left(\frac{y_n\sqrt{D} + \varepsilon_n N f_n\sqrt{d_n}}{q(z_n - t_n\sqrt{D})}\right) \\
&\le H(y_n\sqrt{D} + \varepsilon_n N f_n\sqrt{d_n})\, H(q(z_n - t_n\sqrt{D})) \\
&\le 4 H(y_n\sqrt{D}) H(\varepsilon_n N f_n\sqrt{d_n}) H(q) H(z_n) H(-t_n\sqrt{D}) \\
&\le 4 H(y_n)H(\sqrt{D})H(\varepsilon_n)H(N)H(f_n) H(\sqrt{d_n})H(q)H(z_n)H(t_n)H(\sqrt{D}) \\
&\le 4Dq|N| f_n H(y_n)H(z_n)H(t_n) H(\sqrt{d_n}).
\end{aligned}
$$

By Lemma \ref{lema_1}, there exists a constant $C>0$ such that $|y_n| \le C f_n\sqrt{|d_n|}$, $|z_n| \le C f_n\sqrt{|d_n|}$, and $|t_n| \le C f_n\sqrt{|d_n|}$ for all sufficiently large $n$. It then follows that $H(y_n)=|y_n|\leq C f_n\sqrt{|d_n|}$, $H(z_n)=|z_n|\leq C f_n\sqrt{|d_n|}$, and $H(t_n)=|t_n|\leq C f_n\sqrt{|d_n|}$. Therefore 
	\begin{equation}\label{eq:altura_1}
	H(\tau_n) \le 4D q |N|f_n (C f_n \sqrt{|d_n|})^3 \sqrt{|d_n|} = 4Dq|N| C^3 f_n^4 |d_n|^{2}. 
	\end{equation}

Now, we will use this result to establish a lower bound for the distance between $\tau$ and $\tau_n$ in terms of $|d_n|$.

Applying Liouville's inequality with $r\leq2$, we obtain 
$$
\begin{aligned}
|\tau_n -\tau| &\geq \frac{1}{(2H(\tau) H(\tau_n))^2} \\
&\geq \frac{1}{(8H(\tau)D q |N|C^3 f_n^4 |d_n|^2)^2} \\
&\geq \frac{1}{64D^2q^2 N^2 C^6 H(\tau)^2 f_n^8|d_n|^{4}} \\
&\geq \frac{1}{64D^2q^2 N^2 C^6 H(\tau)^2 |\operatorname{disc}(P_n)|^{4}}. 
\end{aligned}
$$

Therefore, there exists a strictly positive constant $\kappa := \frac{1}{64D^2 q^2 N^2 C^6 H(\tau)^2}$ (depending only on $D$, $q$, $\tau$, $N$ and $C$) such that 
$$|\tau_n -\tau| \geq \frac{\kappa}{|\operatorname{disc}(P_n)|^{4}},$$
as desired. 
\end{proof}


\begin{thebibliography}{99}

\bibitem[AB04]{Alsina-Bayer} M. Alsina and P. B\'ayer, {\it Quaternion orders, quadratic forms, and Shimura curves}, CRM Monograph Series, 22, Amer. Math. Soc., Providence, RI, 2004; MR2038122.

\bibitem[Bak75]{Baker} A. Baker, {\it Transcendental number theory}, Cambridge Univ. Press, London-New York, 1975; MR0422171.

\bibitem[BIR08]{Baker-Il-Rumely} M.~H. Baker, S.-I. Ih and R.~S. Rumely, A finiteness property of torsion points, Algebra Number Theory {\bf 2} (2008), no.~2, 217--248; MR2377370.

\bibitem[BG06]{Bombieri} E. Bombieri and W. Gubler, {\it Heights in Diophantine geometry}, New Mathematical Monographs, 4, Cambridge Univ. Press, Cambridge, 2006; MR2216774.

\bibitem[CS86]{Cornell-Silverman} G. Cornell and J.~H. Silverman, {\it Arithmetic geometry}, Springer-Verlag, New York, 1986; MR0861969.

\bibitem[Cox89]{Cox} D.~A. Cox, {\it Primes of the form $x^2 + ny^2$}, A Wiley-Interscience Publication, Wiley, New York, 1989; MR1028322.

\bibitem[DH09]{David-Hirata} S. David and N. Hirata-Kohno, Linear forms in elliptic logarithms, J. Reine Angew. Math. {\bf 628} (2009), 37--89; MR2503235.

\bibitem[EG15]{Evertse-Gyoy} J.-H. Evertse and K. Gy\H ory, {\it Unit equations in Diophantine number theory}, Cambridge Studies in Advanced Mathematics, 146, Cambridge Univ. Press, Cambridge, 2015; MR3524535.

\bibitem[Hab15]{Habegger} P. Habegger, Singular moduli that are algebraic units, Algebra Number Theory {\bf 9} (2015), no.~7, 1515--1524; MR3404647.

\bibitem[Hab14]{Habegger2} P. Habegger, The Tate-Voloch conjecture in a power of a modular curve, Int. Math. Res. Not. IMRN {\bf 2014}, no.~12, 3303--3339; MR3217663.

\bibitem[HMR24]{Herrero_Menares} S. Herrero-Miranda, R. Menares and J. Rivera-Letelier, There are at most finitely many singular moduli that are $S$-units, Compos. Math. {\bf 160} (2024), no.~4, 732--770; MR4713026.

\bibitem[Kat92]{Katok} S. Katok, {\it Fuchsian groups}, Chicago Lectures in Mathematics, Univ. Chicago Press, Chicago, IL, 1992; MR1177168.

\bibitem[Lan82]{Lang} S. Lang, {\it Introduction to algebraic and abelian functions}, second edition, Graduate Texts in Mathematics, 89, Springer, New York-Berlin, 1982; MR0681120.

\bibitem[Mol12]{Molina} S. Molina, Specialization of heegner points and applications, Universitat Politècnica de Catalunya, 2012.

\bibitem[Mum70]{Mumford} D. Mumford, \emph{Abelian Varieties}, Oxford University Press, 1970.

\bibitem[Qiu23]{Qiu} C. Qiu, The Manin-Mumford conjecture and the Tate-Voloch conjecture for a product of Siegel moduli spaces, Algebra Number Theory {\bf 17} (2023), no.~5, 981--1016; MR4585351.

\bibitem[Shi94]{Shimura} G. Shimura, {\it Introduction to the arithmetic theory of automorphic functions}, reprint of the 1971 original, Publications of the Mathematical Society of Japan Kan\^o{} Memorial Lectures, 11 1, Princeton Univ. Press, Princeton, NJ, 1994; MR1291394.

\bibitem[Sil09]{Silverman-1} J.~H. Silverman, {\it The arithmetic of elliptic curves}, second edition, Graduate Texts in Mathematics, 106, Springer, Dordrecht, 2009; MR2514094.

\bibitem[TV96]{Tate-Voloch}J.~T. Tate and J.~F. Voloch, Linear forms in $p$-adic roots of unity, Internat. Math. Res. Notices {\bf 1996}, no.~12, 589--601; MR1405976.

\bibitem[Voi21]{Voight} J.~M. Voight, {\it Quaternion algebras}, Graduate Texts in Mathematics, 288, Springer, Cham, [2021] \copyright 2021; MR4279905.

\bibitem[Vig80]{Vigneras} M.-F. Vignéras, \emph{Arithmétique des algèbres de quaternions}, Lecture Notes in Mathematics 800, Springer, 1980.

\end{thebibliography}
\end{document}